\documentclass[review,onefignum,onetabnum]{siamart190516}

\usepackage{amsfonts}
\usepackage{epstopdf}
\usepackage{amsmath,wasysym}
\usepackage{amsfonts,pifont}
\usepackage{amssymb}
\usepackage{graphicx,subfigure,color}%
\usepackage{algorithm}
\usepackage{algpseudocode}
\usepackage{multirow}
\usepackage{float}
\usepackage{mathtools,booktabs,cite,upgreek}

\newtheorem{example}{\textit{Example}}
\newtheorem{assumption}{Assumption}

\newcommand{\IR}{{\mathbb{R}}}

\newcommand{\IZ}{{\mathbb{Z}}}
\newcommand{\mO}{{\mathcal{O}}}

\newcommand{\tT}{\intercal}

\newcommand{\bx}{{\bf x}}
\newcommand{\bb}{{\bf b}}
\newcommand{\by}{{\bf y}}
\newcommand{\bu}{{\bf u}}
\newcommand{\bA}{{\bf A}}
\newcommand{\bP}{{\bf P}}
\newcommand{\bB}{{\bf B}}
\newcommand{\bS}{{{\bf S}}}
\newcommand{\bQ}{{{\bf Q}}}
\newcommand{\bG}{{{\bf G}}}
\newcommand{\bz}{{{\bf z}}}
\newcommand{\gafa}{g^{(\alpha)} }
\newcommand{\Gafa}{G^{(\alpha)} }
\newcommand{\btheta}{{\boldsymbol \theta} }
\newcommand{\bafa}{{\boldsymbol \alpha} }
\newcommand{\cafai}{c(\alpha_{i})}
\newcommand{\cafa}{c(\alpha)}

\ifpdf
  \DeclareGraphicsExtensions{.eps,.pdf,.png,.jpg}
\else
  \DeclareGraphicsExtensions{.eps}
\fi

\newsiamremark{remark}{Remark}
\newsiamremark{hypothesis}{Hypothesis}
\crefname{hypothesis}{Hypothesis}{Hypotheses}
\newsiamthm{claim}{Claim}

\headers{Spectral Analysis for Preconditioning of Multi-dimensional RFDE}{Xin Huang, Xue-Lei Lin, Michael K. Ng, and Hai-Wei Sun}

\title{Spectral Analysis for Preconditioning of Multi-dimensional Riesz Fractional Diffusion Equations\thanks{Submitted to the editors DATE.
\funding{This work is supported by research grants of the Science and Technology Development Fund, Macau SAR (file no. 0118/2018/A3), MYRG2018-00015-FST from University of Macau, and
the HKRGC GRF 12306616,
12200317, 12300218, 12300519 and 17201020.}}}

\author{Xin Huang\thanks{Department of Mathematics, University of Macau,
Macao (\email{hxin.ning@qq.com}).}
	\and
Xue-Lei Lin\thanks{Department of Mathematics, University of Macau, Macao (\email{\mbox{hxuellin@gmail.com}}).}
    \and
Michael K. Ng\thanks{Corresponding author. Department of Mathematics, The University of Hongkong, Hong Kong (\email{\mbox{mng@maths.hku.hk}}).}
    \and
	Hai-Wei Sun\thanks{Department of Mathematics, University of Macau, Macao (\email{\mbox{hsun@um.edu.mo}}).}}

\usepackage{amsopn}

\ifpdf
\hypersetup{
  pdftitle={Spectral Analysis for Preconditioning of Multi-dimensional Riesz Fractional Diffusion Equations},
  pdfauthor={X. Huang, X. Lin, Michael K. Ng, and H. Sun}
}
\fi

\begin{document}

\maketitle

\begin{abstract}
In this paper, we analyze the spectra of the preconditioned matrices arising from
discretized multi-dimensional Riesz spatial fractional diffusion equations.
The finite difference method is employed to approximate the multi-dimensional Riesz fractional derivatives, which will generate symmetric positive definite ill-conditioned multi-level Toeplitz matrices. The preconditioned conjugate gradient method with a preconditioner based on the sine transform is employed to solve the resulting linear system. Theoretically, we prove that the spectra of the preconditioned matrices are uniformly bounded
in the open interval $(1/2,3/2)$ and thus the preconditioned conjugate gradient method converges linearly. The proposed method can be extended to multi-level Toeplitz matrices generated by functions with zeros of fractional order.
Our theoretical results fill in a vacancy in the literature. Numerical examples
are presented to demonstrate our new theoretical results in the literature and show the
convergence performance of the proposed preconditioner that is better than  other existing
preconditioners.
\end{abstract}



\section{Introduction}

In this paper, we study the preconditioning technique for the following multi-dimensional Riesz fractional diffusion equations
\begin{equation}\label{model-equation}
-\sum\limits_{i=1}^{m}d_i\frac{\partial^{\alpha_{i}}u(\bf x)}{\partial|x_{i}|^{\alpha_{i}}}=y({\bf x}),\quad {\bf x}\in {\Omega}=\prod_{i=1}^m[a_{i},b_{i}]\subset \IR^m,
\end{equation}
subject to the boundary condition
\begin{equation*}
u({\bf x})=0, \ \bx\in\partial\Omega,
\end{equation*}
where $d_{i}>0$ for $i=1,\dots,m$,   ${\bf x}=(x_{1},\dots,x_{m})\in \IR^m$, $y({\bf x}): \IR^m\mapsto\IR$ is the source term, and $\frac{\partial^{\alpha_i}u({\bf x})}{\partial|x_i|^{\alpha_i}}$ is the Riesz fractional derivative of ${\alpha_i}\in(1,2)$ with respect to $x_i$ defined by
\begin{equation}\label{Riesz}
\frac{\partial^{\alpha_i}u({\bf x})}{\partial|x_i|^{\alpha_i}}=\cafai \left(_{a_i}D_{x_i}^{\alpha_i}u({\bf x})+_{x_i}D_{b_i}^{\alpha_i}u({\bf x})\right),\
\cafai=\frac{-1}{2\cos(\frac{\alpha_{i}\pi}{2})}>0.
\end{equation}
The above left and right Riemann-Liouville (RL) fractional derivatives are defined by
\begin{equation}
_{a_i}{\rm D}_{x_i}^{\alpha_i}u({\bf x})=\frac{1}{\Gamma(2-{\alpha_i})}\frac{\partial^{2}}{\partial x_{i}^{2}}\int_{a_i}^{x_i}\frac{u(x_{1},x_{2},\dots,x_{i-1},\xi,x_{i+1},\dots,x_{m})}{(x_{i}-\xi)^{\alpha_{i}-1}}{\rm d}\xi,
\end{equation}
\begin{equation}\label{right-LR}
_{x_i}{\rm D}_{b_i}^{\alpha_i}u({\bf x})=\frac{1}{\Gamma(2-{\alpha_i})}\frac{\partial^{2}}{\partial x_{i}^{2}}\int_{x_i}^{b_i}\frac{u(x_{1},x_{2},\dots,x_{i-1},\xi,x_{i+1},\dots,x_{m})}{(\xi-x_i)^{\alpha_i-1}}{\rm d}\xi,
\end{equation}
respectively, where ${\Gamma(\cdot)}$ is the gamma function.

Fractional calculus has received an increasing interest since its applications involve various fields  including physics, chemistry, engineering; see \cite{CW-Springer-2001,GS-Toepliz-form,H-application-physics-2000,P-FDE-1999}. The Riesz fractional derivative, which derives from the kinetic of chaotic dynamics \cite{AG-Chaos-1997}, is generalized as one of the most popular fractional calculus. Recently, the study of the Riesz fractional derivative has been urgent and significant as it can be applied to lattice model with long-range interactions \cite{ZD-Riesz-2012}, nonlocal dynamics \cite{T-Riesz-appication-2010}, and so on.

After discretization by the finite difference method, the resulting coefficient matrix of the above multi-dimensional Riesz fractional derivative in \eqref{model-equation} is a dense
multi-level Toeplitz matrix; i.e.,
each block has a Toeplitz structure. It is interesting to note that such multi-level
Toeplitz matrix can be generated by a continuous real-valued even function which is
nonnegative defined on the interval $[-\pi,\pi]$ \cite{PS-JSC-2016}. Moreover,
the diagonal entries are the Fourier coefficients of the generating function
with ${\alpha}$-th order zero at the origin ($1 < \alpha < 2$). Hence, the condition number of
the discretized linear system is unbounded as the matrix size tends to infinity.  More precisely, the condition number grows as $n^{\alpha}$, where $n$ denotes the matrix size; see \cite{RM-band-1991,BG-LAA-1998,Serra-LAA-1998}. Therefore, the resulting linear system arising from \eqref{model-equation} is ill-conditioned and thus the conjugate gradient (CG) method for this system converges slowly.

In order to speed up the convergence of the CG method, preconditioning techniques have
been proposed and developed for ill-conditioned Toeplitz systems. For examples,
banded Toeplitz preconditioners \cite{RM-band-1991,BFS-CMA-1993} were
proposed to handle ill-conditioned Toeplitz linear systems where the generating functions of Toeplitz matrices have zeros of even order.
When these banded Toeplitz preconditioners are applied to
Toeplitz matrices generated by functions
with ${\alpha}$-th order zero at the origin ($1 < \alpha < 2$),
the condition numbers of these preconditioned systems are not uniformly
bounded.

Besides, several strategies have been exploited for the ill-conditioned Toeplitz systems, such as $\tau$-preconditioners (which can be diagonalized by the discrete sine transform matrix) \cite{BB-ACM-1990,BD-SISC-1997,CNW-LAA-1996,S-MC-1999},  circulant preconditioners \cite{PS-BIT-1999,CYM-NA-2000,PS-SISC-2001,NSV-TCS-2004}, and multigrid methods \cite{FS-CAL-1991,FS-SISC-1996,RS-JSC-1998,SJC-BIT-2001,BFS-CMA-1993,Serra-Calcolo-1995}.  These approaches can significantly speed up the convergence
of iterative methods for solving Toeplitz systems where their
generating functions have an
${\alpha}$-th order zero at the origin ($1 < \alpha < 2$).
However, the linear convergence of these methods
cannot be {\it theoretically} confirmed
for solving such Toeplitz systems.
Lately, some efficient preconditioners are developed to solve
linear systems arising from fractional diffusion equations;
see \cite{DKMT-AdvCM-2020,LYJ-CCP-2015,DMS-JCP-2016, BEV-ArXiv-2019,MDDM-JCP-2017}. Nevertheless, from the theoretical point of view, the spectra of these preconditioned matrices are not shown to
be bounded independent of the matrix sizes
and hence the linear convergence cannot be guaranteed.

In order to tackle this theoretical problem, Noutsos, Serra and Vassalos
\cite{NS-2016-LAA} exploited a multiple step preconditioning and applied to the case where the generating functions of coefficient matrices have fractional order zeros.
In their method, they proposed a new
$\tau$-preconditioner that is constructed from the generating function of the given
Toeplitz matrix. Numerical results were shown that their method worked
very well for the Toeplitz matrices whose generating functions have
fractional order zeros.
Theoretically, they have proved that the largest eigenvalue
of the preconditioned matrix has an upper bound independent of the matrix size. Nevertheless, it is still unclear whether the smallest eigenvalue of the preconditioned matrix
is bounded below away from zero; see the remark in \cite{NS-2016-LAA}.


The main aim of this paper is to conduct the spectral analysis of the
$\tau$-preconditioner for the ill-conditioned multi-level Toeplitz system arising from the discretized Riesz fractional derivatives.
Theoretically, we prove that the spectra of the $\tau$-preconditioned matrices are uniformly bounded
in the open interval $(1/2,3/2)$ and thus the preconditioned CG (PCG)
method converges linearly. Furthermore, the proposed method can be extended to multi-level Toeplitz matrices generated by functions with zeros of fractional order. Similarly, we show that the spectra
of these preconditioned multi-level Toeplitz matrices are uniformly bounded. Numerical examples
are presented to verify our new theoretical results in the literature and show
the good performance of the proposed preconditioner.

The outline of the rest paper is as follows. In \Cref{multi-level_toeplitzsystem},
multi-level Toeplitz matrices are generated.
In \Cref{preconditioner1}, the preconditioner is proposed and developed. In \Cref{spectrumanalysis},
  the spectral analysis of the preconditioned matrix is discussed.
In \Cref{extension}, a new preconditioning technique for multi-level Toeplitz matrices
is studied.
Numerical experiments are given in \Cref{numericalresults} to show the
performance of the proposed preconditioner. Finally, some concluding remarks are given in \Cref{conclusion}.
\section{Multi-level Toeplitz matrices}\label{multi-level_toeplitzsystem}
A matrix, whose entries are constant along the diagonals with the following form
\begin{equation}\label{toeplitz-matrix}
T_{n}=\left[
\begin{array}{ccccc}
t_{0}   & t_{-1}  &\dots  & t_{2-n} & t_{1-n} \\
t_{1}   & t_{0}   &t_{-1} &\dots    & t_{2-n} \\
\vdots  & t_{1}   & t_{0} &\ddots   & \vdots  \\
t_{n-2} & \ddots &\ddots & \ddots  & t_{-1}  \\
t_{n-1} & t_{n-2} & \dots & t_{1}   & t_{0}   \\
\end{array}
\right],
\end{equation}
is called a Toeplitz matrix. Assume that the diagonals $\{t_{k}\}_{k=1-n}^{n-1}$ of $T_{n}$ are the Fourier coefficients of a function $f$; i.e.,
\begin{equation*}
t_{k}=\frac{1}{2\pi}\int_{-\pi}^{\pi}f(\theta)e^{-{\bf i}k\theta}d\theta,
\end{equation*}
then the function $f$ is called the generating function of $T_{n}$. Generally, denote $T_{n}=T_{n}(f)$ to emphasize that an $n\times n$ Toeplitz matrix $T_{n}$ is generated by $f$. Moreover, if the generating function is real  and even, then $T_{n}$ is real symmetric for all $n$.

Let ${\btheta}=(\theta_{1},\theta_{2},\dots,\theta_{m})\in [-\pi,\pi]^{m}$, $f(\btheta )=f(\theta_{1},\dots,\theta_{m})\in L^{1}([-\pi,\pi]^{m})$. The matrix generated by function $f({\btheta})$ is an $m$-level Toeplitz matrix with Toeplitz structure on each level.
Let $n_{i}$ for $i=1,2,\dots,m$ be positive integers. Denote $N=\prod\limits_{i=1}^{m}n_{i}$. Then an $m$-level Toeplitz matrix with the size $N\times N$ is a block Toeplitz matrix with Toeplitz block, and its form is
\begin{equation}\label{m-level-matrix}
T_{N}^{m}=
\left[
\begin{array}{cccc}
T_{0}^{m-1}       & T_{-1}^{m-1} & \cdots      & T_{1-{n_m}}^{m-1} \\
T_{1}^{m-1}       & T_{0}^{m-1}  & \ddots      & \vdots             \\
\vdots            & \ddots       & \ddots      & T_{-1}^{m-1}        \\
T_{n_{m}-1}^{m-1} & \cdots       & T_{1}^{m-1} & T_{0}^{m-1}
\end{array}
\right],
\end{equation}
where each block $T_{j}^{m-1}$ for $j=0,\dots,n_{m}-1$ is an $(m-1)$-level block Toeplitz of size $(n_{1}\cdots n_{m-1})\times(n_{1}\cdots n_{m-1})$ with $(m-2)$-level Toeplitz block of size $(n_{1}\cdots n_{m-2})\times(n_{1}\cdots n_{m-2})$, and so on. Let $(j_{1},\dots,j_{m})\in {\IZ}^{m}$ be a multi-index. Then the coefficients of $T_{N}^{m}$ can be obtained by \cite{P-JMA-2019}
\begin{equation}\label{Fouriercoefficient}
t_{j_{1},j_{2},\dots,j_{m}}=\frac{1}{(2\pi)^{m}}\int_{[-\pi,\pi]^{m}}f({\btheta})e^{-{\bf i}\sum\limits_{i=1}^{m}\theta_{i}j_{i}} d{\btheta}.
\end{equation}

\subsection{Discretized 1D Riesz fractional derivative}

Firstly, we consider the one dimensional case.
Let $n$ be the partition of the interval $[a,b]$. We define a uniform spatial partition
\begin{equation*}
h=\frac{b-a}{n+1},\quad \eta_{j}=a+jh,\quad {\rm for}\ j=0,1,\dots,n+1.
\end{equation*}
Then the following shifted Gr\"{u}nwald-Letnikov formula is exploited to approximate the left- and right- RL fractional derivatives at grid point $\eta_{j}$,
\begin{equation}\label{L-RL}
_{a}D_{x}^{\alpha}u(\eta_{j})=\frac{1}{h^{\alpha}}\sum_{k=0}^{j+1}g_{k}^{(\alpha)}u(\eta_{j-k+1})+\mO(h),
\end{equation}
\begin{equation}\label{R-RL}
_{x}D_{b}^{\alpha}u(\eta_{j})=\frac{1}{h^{\alpha}}\sum_{k=0}^{n-j+2}g_{k}^{(\alpha)}u(\eta_{j+k-1})+\mO(h),
\end{equation}
where the coefficients $g_{k}^{(\alpha)}$ are defined by
\begin{equation}\label{gg}
g_{0}^{(\alpha)}=1,\quad g_{k}^{(\alpha)}=\left(1-\frac{\alpha+1}{k}\right)g_{k-1}^{(\alpha)},\quad {\rm for}\ k\geq 1.
\end{equation}
It can be shown that the coefficients $g_{k}^{(\alpha)}$ defined above have the following properties.

\begin{lemma}\label{g-property}
{\rm (see \cite{MT-JCAM-2004})} For ${\alpha}\in(1,2)$, the coefficients $g_{k}^{(\alpha)}$, $k=0,1,\dots$, satisfy
\begin{displaymath}
\left\{
\begin{aligned}
&g_{0}^{(\alpha)}=1, \ g_{1}^{(\alpha)}=-{\alpha}<0, \ g_{2}^{(\alpha)}>g_{3}^{(\alpha)}>\cdots>0,\\
&\sum\limits_{k=0}^{\infty}g_{k}^{(\alpha)}=0,\ \sum\limits_{k=0}^{n}g_{k}^{(\alpha)}<0,\ {\rm for}\ n\geq 1.
\end{aligned}
\right.
\end{displaymath}
\end{lemma}

By applying \eqref{L-RL} and \eqref{R-RL} to the model equation \eqref{model-equation}, we obtain
\begin{equation}
-d\frac{\cafa}{h^{\alpha}}\left(\sum_{k=0}^{j+1}g_{k}^{(\alpha)}u(\eta_{j-k+1})+
\sum_{k=0}^{n-j+2}g_{k}^{(\alpha)}u(\eta_{j+k-1})\right)=y(\eta_{j})+\mO(h).
\end{equation}
Defining $u_{j}$ as the numerical approximation of $u(\eta_{j})$, setting $y_{j}=y(\eta_{j})$, and omitting the small term $\mO(h)$, the finite difference scheme for solving \eqref{model-equation} is constructed as follows
\begin{equation}\label{numerical scheme}
-d\frac{\cafa}{h^{\alpha}}\left(\sum_{k=0}^{j+1}g_{k}^{(\alpha)}u_{j-k+1}+
\sum_{k=0}^{n-j+2}g_{k}^{(\alpha)}u_{j+k-1}\right)=y_{j}, \quad {\rm for} \ j=1,\dots,n.
\end{equation}

Let $u=[u_{1},\dots,u_{n}]^{\tT}$, $y=[y_{1},\dots,y_{n}]^{\tT}$. Then, the numerical scheme \eqref{numerical scheme} can be simplified as the following matrix-vector form
\begin{equation}\label{linearsystem}
A_nu=y,
\end{equation}
with $A_n=w G_{n}^{(\alpha)}$, where $w=\frac{dc(\alpha)}{h^{\alpha}}>0$ and
\begin{equation}{\small
 G_{n}^{(\alpha)}=-\left[
\begin{array}{cccccc}
2g_{1}^{(\alpha)} & g_{0}^{(\alpha)}+g_{2}^{(\alpha)} & g_{3}^{(\alpha)} &\ddots &g_{n-1}^{(\alpha)} &g_{n}^{(\alpha)} \\
 g_{0}^{(\alpha)}+g_{2}^{(\alpha)} &2g_{1}^{(\alpha)} &g_{0}^{(\alpha)}+g_{2}^{(\alpha)} &g_{3}^{(\alpha)} &\ddots  &g_{n-1}^{(\alpha)}  \\
 \vdots  & g_{0}^{(\alpha)}+g_{2}^{(\alpha)} &2g_{1}^{(\alpha)} & \ddots & \ddots & \vdots \\
    \vdots  & \ddots    & \ddots & \ddots & \ddots &  g_{3}^{(\alpha)}\\
 g_{n-1}^{(\alpha)}  & \ddots & \ddots &\ddots & 2g_{1}^{(\alpha)} & g_{0}^{(\alpha)}+g_{2}^{(\alpha)} \\
g_{n}^{(\alpha)} & g_{n-1}^{(\alpha)}  & \cdots & \cdots & g_{0}^{(\alpha)}+g_{2}^{(\alpha)} &2g_{1}^{(\alpha)}\\
    \end{array}
  \right]. \label{A} }
\end{equation}
It is evident that the matrix $G_{n}^{(\alpha)}$ is a symmetric Toeplitz matrix, which is generated by an integrable real-value even function defined on the interval $[-\pi,\pi]$. In the following, the generating function of matrix $G_{n}^{(\alpha)}$ will be presented.

\begin{lemma}\label{generating-function}
{\rm(see \cite{PS-JSC-2016})} The generating function of matrix $G_{n}^{(\alpha)}$ is
\begin{equation}\label{g-form}
g_\alpha({\theta})=\left\{
\begin{array}{l} \vspace{1mm}
-2^{\alpha+1}(\sin\frac{-{\theta}}{2})^{\alpha}\cos[\frac{\alpha}{2}(\pi+{\theta})-{\theta}], \ {\theta}\in [-\pi,0), \\
-2^{\alpha+1}(\sin\frac{{\theta}}{2})^{\alpha}\cos[\frac{\alpha}{2}(\pi-{\theta})+{\theta}],\ \ {\theta}\in [0,\pi].
\end{array}
\right.
\end{equation}
\end{lemma}

By \cref{generating-function},   the generating function of $G_{n}^{(\alpha)}$ possesses the following properties, which will be exploited to further study.

\begin{lemma}\label{fractional-zeros-bounded}
For any ${\alpha}\in (1,2)$, it holds that
\begin{equation*}
\frac{1}{2}\le \frac{|\theta|^{\alpha}}{g_\alpha({\theta})}\le\frac{\pi^{2}}{-8\cos(\frac{\pi\alpha}{2})},\ where\ {\theta}\in[-\pi,\pi].
\end{equation*}
\end{lemma}
\begin{proof}
We only consider the case that ${\theta}\in (0,\pi]$ as it is an even function. From \cref{generating-function}, we have
\begin{equation*}
\frac{|{\theta}|^{\alpha}}{g_\alpha({\theta})}=-\frac{{\theta}^{\alpha}}{2^{\alpha+1}(\sin\frac{{\theta}}{2})
^{\alpha}\cos[\frac{\alpha}{2}(\pi-{\theta})+{\theta}]}.
\end{equation*}
Note that $1<{\alpha}<2$, it holds
\begin{equation*}
\frac{1}{2}<\frac{{\theta}^{\alpha}}{2^{\alpha+1}(\sin\frac{{\theta}}{2})^{\alpha}}=\frac{1}{2}
\frac{(\frac{{\theta}}{2})^{\alpha}}{(\sin\frac{{\theta}}{2})^{\alpha}}\le\frac{\pi^{2}}{8},
\end{equation*}
and
\begin{equation*}
1\le-\frac{1}{\cos[\frac{\alpha}{2}(\pi-{\theta})+{\theta}]}\le -\frac{1}{\cos(\frac{\pi\alpha}{2})}.
\end{equation*}
Similarly, we can derive the same conclusion for ${\theta}\in [-\pi,0)$ and the result is concluded.
\end{proof}


In the light of the definition of the fractional order zero defined in \cite{DMS-JCP-2016}, we deduce that the generating function $g_\alpha({\theta})$ has a zero of order ${\alpha}$ at ${\theta}=0$.

\subsection{Multi-dimensional Riesz fractional derivatives}
Now, we extend our investigation to the multi-dimensional cases as in \eqref{model-equation}.
To obtain the discretized form of multi-dimensional Riesz fractional diffusion equations, some notations are required.

Denote $I_{k}$ be a $k\times k$ identity matrix. Let $n_{i}^{-}=\prod\limits_{j=1}^{i-1}n_{j}$ and $n_{i}^{+}=\prod\limits_{j=i+1}^{m}n_{j}$ for $i=1,2,\dots,m$, respectively.
In particular, take $n_{1}^{-}=n_{m}^{+}=1$. Let $h_{i}=\frac{b_{i}-a_{i}}{n_{i}+1}$, for $i=1,\dots,m$, $\eta^i_j=a_i+jh_i$, and $y_{j_1,j_2,\ldots,j_m}=y(\eta^1_{j_1}, \eta^2_{j_2},\ldots,\eta^m_{j_m})$. Denote
\begin{equation*}
\bu=[u_{1,1,\ldots,1},\ldots,u_{n_1,1,\ldots,1},u_{1,2,\ldots,1},\ldots,\ldots,u_{n_1,n_2,\ldots,n_m}]^\tT\in \IR^N
\end{equation*}
and
\begin{equation*}
\by=[y_{1,1,\ldots,1},\ldots,y_{n_1,1,\ldots,1},y_{1,2,\ldots,1},\ldots,\ldots,y_{n_1,n_2,\ldots,n_m}]^\tT \in \IR^N.
\end{equation*}
Analogously, using the shifted Gr\"{u}nwald-Letnikov formula to discretize the multi-dimensional Riesz fractional diffusion equations \eqref{model-equation}, we obtain the matrix-vector form of the resulting linear system as
\begin{equation}\label{high-linearsystem}
 \bA\bu=\by,
\end{equation}
where
 \begin{equation}\label{high-A}
 {\bf A}=\sum_{i=1}^{m}I_{n_{i}^{-}}\otimes {A}_{n_i}\otimes I_{n_{i}^{+}},
 \end{equation}
in which ${A}_{n_i}=w_{i}G_{n_i}^{(\alpha_i)}$,  $w_{i}=\frac{d_{i}c({\alpha_i})}{h_{i}^{\alpha_i}}>0$, and $G_{n_i}^{(\alpha_i)}$ is defined as in \eqref{A}.
Moreover, the generating function of $\bA$ is as follows,
\begin{equation}\label{multi-functions}
f_{\bafa}(\btheta)=\sum_{i=1}^m w_ig_{\alpha_i}(\theta_i),
\end{equation}
where $\bafa=(\alpha_1,\ldots,\alpha_m)$.

Note that the generating function $g_{\alpha_i}({\theta_i})$ is nonnegative and hence so is $f_{\bafa}({\btheta})$, which manifests that the coefficient matrices $A_{n_i}$ and ${\bf A}$ are both symmetric positive definite. It is well-known that the CG method has been recognized as the most appropriate method for solving symmetric positive definite Toeplitz systems. Actually, since the generating functions have zeros, the corresponding matrices are  ill-conditioned, which will slow down the convergent rate of the CG method. In order to speed up the convergent rate, the preconditioning technique is employed to solve the ill-conditioned system. 
\section{Sine transform based preconditioners}\label{preconditioner1}
Let $T_{n}$ be a given symmetric Toeplitz matrix whose first column is $[t_{0},t_{1},\dots,t_{n-1}]^{\tT}$. Then, the natural ${\tau}$ matrix ${\tau}(T_{n})$ of $T_{n}$ can be determined by the Hankel correction\cite{BB-ACM-1990}
\begin{equation}\label{tauform}
\tau(T_{n})=T_{n}-H_{n},
\end{equation}
where $H_{n}$ is a Hankel matrix whose entries are constants along the antidiagonals, in which the antidiagonals are given as
\begin{equation*}
[t_{2},t_{3},\dots,t_{n-1},0,0,0,t_{n-1},\dots,t_{3},t_{2}]^{\tT}.
\end{equation*}
More preciously, the entries $p_{ij}$ of ${\tau}(T_{n})$ can be generalized as
\begin{equation}\label{tau-element}
p_{ij}=\left\{
\begin{aligned}
&t_{|j-i|}-t_{i+j}, && i+j<n-1,  \\
&t_{|j-i|}, && i+j=n-1,n,n+1, \\
&t_{|j-i|}-t_{2n+2-(i+j)}, && {\rm otherwise}.
\end{aligned}
\right.
\end{equation}

Notice that a ${\tau}$ matrix can be diagonalized by the sine transform matrix $S_n$   \cite{BC-LAA-1983}; i.e.,
\begin{equation}\label{diagonalized}
\tau(T_{n})=S_n{\Lambda_n}S_n,
\end{equation}
where $\Lambda_n$ is a diagonal matrix holding all the eigenvalues of $\tau(T_{n})$ and the entries of $S_n$ are given by
\begin{equation}\label{sine matrix element}
[S_n]_{j,k}=\sqrt{\frac{2}{n+1}}\sin\left(\frac{\pi j k}{n+1}\right), \ 1\leq k,j\leq n.
\end{equation}
Obviously, $S_n$ is a symmetric orthogonal matrix, and the matrix-vector multiplication $S_nv$ for any vector $v$ can be computed with only ${\mO}(n\log n)$ operations by the fast sine transform. Likewise, the matrix-vector product ${\tau}(T_{n})^{-1}v=S_n{\Lambda}_{n}^{-1}S_nv$ can be done in ${\mO(n \log n)}$ operations. Moreover, the eigenvalues of the ${\tau}$ matrix can be determined by its first column. Thus, $\mO(n)$ storage and $\mO({n\log n})$ computational complexity are required for saving and computing the eigenvalues of the ${\tau}$ matrix, respectively.

Now we consider the preconditioner matrix of the linear system (\ref{linearsystem}).
Recalling the form of the  coefficient matrix, we obtain the preconditioner as
\begin{equation}\label{preconditioner}
P_n={\tau}(A_n)=w{\tau}(G^{(\alpha)}_n),
\end{equation}
where $G^{(\alpha)}_n$ is the symmetric positive definite Toeplitz matrix defined in \eqref{A}.
For general ${\tau}$ matrix, we obtain its eigenvalues by the following lemma.
\begin{lemma}\label{eigenvalues_of_G}
{\rm(see \cite{BB-ACM-1990})} Let $T_n$ be a  symmetric Toeplitz matrix whose first column is $[t_{0},t_{1},\dots,t_{n-1}]^\tT$. Then the eigenvalues of ${\tau}(T_{n})$ can be expressed as
\begin{equation*}
{\sigma}_{j}=t_0+2\sum^{n-1}_{k=1}t_{k} \cos (j\zeta_k),\quad j=1,\ldots,n,
\end{equation*}
where $\zeta_k=\frac{\pi k}{n+1}$, $k=1,\dots,n-1$.
\end{lemma}

Lemma \ref{eigenvalues_of_G} provides a methodology to compute the eigenvalues of the $\tau$ matrix. Taking advantage of Lemma \ref{eigenvalues_of_G},  the preconditioner $P_n$  can be verified to be positive definite in the following lemma.

\begin{lemma}\label{tau-spd}
The preconditioner $P_n=w{\tau}(G^{(\alpha)}_n)$ defined in \eqref{preconditioner} is symmetric positive definite.
\end{lemma}
\begin{proof}
The first column of $ G^{(\alpha)}_n $ in \eqref{A} is
\begin{equation*}
\left[-2g_{1}^{(\alpha)},-g_{0}^{(\alpha)}-g_{2}^{(\alpha)},-g_{3}^{(\alpha)},\dots, -g_{n}^{(\alpha)}\right]^{\tT}.
\end{equation*}
According to \cref{eigenvalues_of_G}, the $j$th eigenvalue  of $\tau(G^{(\alpha)}_n)$ can be expressed as
 \begin{eqnarray*}
 {\delta}_{j}     &=& -2\gafa_1-2\left(\gafa_0+\gafa_2\right)\cos\left(\frac{\pi j }{n+1}\right)-2\sum_{k=3}^{n}\gafa_k\cos\left(\frac{\pi j(k-1)}{n+1}\right)  \\
     &\ge&  -2\gafa_1 - 2\left(\gafa_0+\gafa_2\right)-2\sum_{k=3}^n \gafa_k   \\
      &=& -2\sum^{n }_{k=0}\gafa_k > 0. \qquad {\rm (by \, \cref{g-property})}
  \end{eqnarray*}
   Therefore, we deduce that $\tau(G^{(\alpha)}_n)$ is symmetric positive definite and conclude that so is $P_n$.
\end{proof}

Making use of \cref{tau-spd}, we derive the following corollary.
\begin{corollary}
The preconditioner $P_n$ defined in \eqref{preconditioner} is invertible.
\end{corollary}


Now we consider the $\tau$-preconditioner for multi-level Toeplitz matrices. Recall the coefficient matrix
\begin{equation*}
{\bf A}=\sum_{i=1}^{m}I_{n_{i}^{-}}\otimes {A}_{n_i}\otimes I_{n_{i}^{+}}.
\end{equation*}
The preconditioner matrix of the linear system (\ref{high-linearsystem}) can be expressed as
\begin{equation}\label{high-preconditioner}
 {\bf P}=\sum_{i=1}^{m}{I}_{n_{i}^{-}}\otimes {\tau({A}_{n_i})}\otimes {I}_{n_{i}^{+}},
\end{equation}
where ${\tau}({A}_{n_i})=w_{i}\tau(G_{n_i}^{(\alpha_i)})$ for $i=1,\dots,m$.
Denote
\begin{equation}\label{Ai}
{A}_{i}={I}_{n_{i}^{-}}\otimes {A}_{n_i}\otimes {I}_{n_{i}^{+}},
\end{equation}
and the corresponding $\tau$-matrix can be defined as ${\tau}_{1}({ A}_{i})={I}_{n_{i}^{-}}\otimes {\tau({A}_{n_i})}\otimes {I}_{n_{i}^{+}}$.
\begin{lemma}
The preconditioner ${\bf P}$ defined in (\ref{high-preconditioner}) is symmetric positive definite.
\end{lemma}
\begin{proof}
Let ${\Lambda}_{n_i}$ be a diagonal matrix holding the eigenvalues of ${\tau}({A}_{n_i})$. Then
  ${\tau}({A}_{n_i})=S_{n_i}{\Lambda_{n_i}}S_{n_i}$.
 \cref{tau-spd} has shown that all the eigenvalues of ${\tau}(G_{n_i}^{(\alpha_i)})$ are positive; that is, ${\tau}({A}_{n_i})$ for $i=1,2\dots,m$ are positive definite. By virtue of the properties of Kronecker product, it is easy to obtain
\begin{equation*}
{I}_{n_{i}^{-}}\otimes {\tau({A}_{n_i})}\otimes {I}_{n_{i}^{+}}=\bS (I_{n_{i}^{-}}\otimes {\Lambda}_{n_i}\otimes I_{n_{i}^{+}})\bS ,
\end{equation*}
where $\bS =\bigotimes\limits^m_{i=1}S_{n_i}$,
which manifests that all the eigenvalues of ${\tau}_{1}(A_{i})$ are positive. As ${\bf P}$ is the summation of symmetric positive definite matrices, it follows that ${\bf P}$ is also symmetric positive definite.
\end{proof}

\section{Spectral analysis for preconditioned matrices}\label{spectrumanalysis}

In this section, we discuss the spectra of the preconditioned matrices.

\subsection{One dimensional case}\label{1spectrum}

In the following, we discuss the spectrum of ${\tau}(\Gafa_n)^{-1}H_n$, which is needed in the theoretical analysis.

First of all, we focus on the matrix ${\tau}(\Gafa_n)=\Gafa_n-H_n$, where $H_n$ is the Hankel matrix as in \eqref{tauform}, and $G_{n}^{(\alpha)}$ is defined in \eqref{A}. Denote the first column of ${\Gafa_n}$ as
\begin{equation}\label{firstcolumnG}
[t_{0},t_{1},\dots,t_{n-1}]^{\tT}=[-2g_{1}^{(\alpha)},-(g_{0}^{(\alpha)}+g_{2}^{(\alpha)}),\dots,-g_{n }^{(\alpha)}]^{\tT}.
\end{equation}
By \cref{g-property}, we have
\begin{lemma}\label{g_property_later}
Let $t_j, \ j=0,1,\ldots,n-1$, be defined in \eqref{firstcolumnG}. Then
\begin{equation*}
t_0>0,\, t_1<t_2<\cdots<t_{n-1}<0,
\end{equation*}
and
\begin{equation*}
t_0+2\sum^{n-1}_{j=1}t_j>0.
\end{equation*}
\end{lemma}
\begin{proof}
Using \eqref{gg} and \cref{g-property}, the results are immediately concluded.
\end{proof}

 Let   $h_{ij}$ and $p_{ij}$ be the entries of  $H_n$ and ${\tau}(\Gafa_n)$, respectively. From Lemma \ref{g_property_later}, we have
 \begin{equation}\label{H-entries}
h_{ij}=\left\{
\begin{array}{ll}
 t_{i+j}, & i+j<n-1,  \\
0,  &i+j=n-1,n,n+1,  \\
t_{2n+2-(i+j)}, &{\rm otherwise},
\end{array}
\right.
\end{equation}
and
\begin{equation}\label{tau-entries}
  p_{ij}=t_{|i-j|}-h_{ij}.
\end{equation}
It is obvious that $h_{ij}\leq 0$ for $1\le i,j\le n$. By \cref{g_property_later}, we will show that $p_{ij}$ possesses the following properties.

\begin{lemma}\label{rm2}
Let $p_{ij}$ be the entries of ${\tau}(G_{n}^{(\alpha)})$ defined in \eqref{tau-entries}. For $1\le i,j\le n$, it holds that
\begin{equation}\label{sign}
 p_{ii}>0, \quad {\rm and}\quad  p_{ij}<0, \, {\rm for}\quad i\neq j.
\end{equation}
\end{lemma}

\begin{proof}
It is evident that $p_{ii}=t_0-h_{ii}>0$ for each $i=1,\dots,n$.
On the other hand, let $1\le i<j\le n$. It is easy to check that
\begin{equation*}
|i-j|<i+j, \ {\rm and} \quad |i-j|=|(n-i)-(n-j)|<2n-(i+j)+2.
\end{equation*}
By \cref{g_property_later} and formulas \eqref{H-entries}--\eqref{tau-entries}, we derive that
\begin{equation*}
p_{ij}=t_{|i-j|}-t_{i+j}(\, {\rm or} \, t_{|i-j|}-t_{2n-(i+j)+2})<0.
\end{equation*}
The proof is complete.
\end{proof}


\begin{lemma}\label{tau-HC}
Let $\Gafa_n$ be the Toeplitz matrix defined in \eqref{A}, and $H_n$ be the corresponding Hankel matrix whose entries are defined in \eqref{H-entries}.  Then, the eigenvalues of $\tau(\Gafa_n)^{-1}H_n$ fall inside the open interval $(-1/2,1/2)$.
\end{lemma}
\begin{proof}
Let ${\lambda}$ be the eigenvalue of ${\tau}(\Gafa_n)^{-1}H_n$, and $z=[z_{1},z_{2},\dots,z_{n}]^{\tT}$ be the corresponding eigenvector with the largest component having the magnitude $1$; i.e., $\max\limits_{1\le j \le n} {|z_j|}=1$. Then, we have
\begin{equation*}
 H_nz={\lambda}{\tau}(\Gafa_n)z.
\end{equation*}
 For each $i$, it holds that
 \begin{equation*}
 \sum_{j=1}^{n}h_{ij}z_{j}={\lambda}\sum\limits_{j=1}^{n}p_{ij}z_{j},
\end{equation*}
 which can be written as
 \begin{equation*}
{\lambda}p_{ii}z_{i}=\sum\limits_{j=1}^{n}h_{ij}z_{j}-{\lambda}\sum\limits_{j=1,j\neq i}^{n}p_{ij}z_{j}.
\end{equation*}
Let $|z_{k}|=1$. Then, by the above formula, it follows
\begin{equation*}
|{\lambda}||p_{kk}|\leq \sum\limits_{j=1}^{n}|h_{kj}|+|{\lambda}|\sum\limits_{j=1,j\neq k}^{n}|p_{kj}|.
\end{equation*}
Accordingly, it is resulted that
\begin{equation*}
|{\lambda}|\leq \frac{\sum\limits_{j=1}^{n}|h_{kj}|}{|p_{kk}|-\sum\limits_{j=1,j\neq
k}^{n}|p_{kj}|}.
\end{equation*}
By \cref{g_property_later} and \cref{rm2}, we have
\begin{align*}
&|p_{kk}|-\sum_{j=1,j\neq k}^{n}|p_{kj}|-2\sum_{j=1}^{n}|h_{kj}| \\
=&(t_0-h_{kk})- \sum_{j=1,j\neq k}^{n}(h_{kj}-t_{|k-j|})+2\sum_{j=1}^{n}h_{kj}
&&{\rm (by\ \eqref{tau-entries} \ and \, \eqref{sign})} \\
=& t_0+\sum_{j=1,j\neq k}^{n}t_{|k-j|}+\sum_{j=1}^{n}h_{kj} \\
=& t_0+\left(\sum_{j=1}^{k-1}t_{j}+\sum_{j=1}^{n-k}t_{j}\right)+\left(\sum_{j=k+1}^{n-1}t_{j}+\sum_{j=n-k+2}^{n-1}t_{j}\right)\\
\geq&t_0+2\sum_{j=1}^{n-1}t_{j}  > 0,
\end{align*}
which manifests $|\lambda|<1/2$. Therefore, we derive that the eigenvalues of ${\tau}(G_{n}^{(\alpha)})^{-1}H_{n}$ fall inside the interval $(-1/2,1/2)$.
\end{proof}

The above lemma indicates that the spectrum of ${\tau}(G_{n}^{(\alpha)})^{-1}H_n$ are bounded, which is the key to obtain the spectrum of ${\tau}(G_{n}^{(\alpha)})^{-1}G_{n}^{(\alpha)}$.

\begin{theorem}\label{tau-spectrum}
Let ${\lambda}({\tau}(G_{n}^{(\alpha)})^{-1}G_{n}^{(\alpha)})$ be the eigenvalues of matrix ${\tau}(G_{n}^{(\alpha)})^{-1}G_{n}^{(\alpha)}$. Then the following inequality holds
\begin{equation*}
\frac{1}{2}<{\lambda}({\tau}(G_{n}^{(\alpha)})^{-1}G_{n}^{(\alpha)})<\frac{3}{2}.
\end{equation*}
\end{theorem}

\begin{proof}
Using
\begin{equation*}
I_n+{\tau}(G_{n}^{(\alpha)})^{-1}H_n={\tau}(G_{n}^{(\alpha)})^{-1}({\tau}(G_{n}^{(\alpha)})+H_n)=
{\tau}(G_{n}^{(\alpha)})^{-1}G_{n}^{(\alpha)},
\end{equation*}
taking advantage of the conclusion obtained in \cref{tau-HC}, the proof is complete.
\end{proof}


Based on \cref{tau-spectrum}, the following corollary which gives an upper bound in terms of the condition number of the preconditioned matrix can be achieved.

\begin{corollary}\label{conditionnumber}
Let $P_{n}$ defined in \eqref{preconditioner} be the preconditioner of the linear system \eqref{linearsystem}.
Then, the condition number of the preconditioned matrix $P_{n}^{-1}A_{n}$ is less than $3$.
\end{corollary}
\begin{proof}
It is easy to verify that
\begin{equation*}
P_{n}^{-1}A_n={\tau}(G_{n}^{(\alpha)})^{-1}{G_{n}^{(\alpha)}}.
\end{equation*}
By \cref{tau-spectrum}, we obtain
\begin{equation*}
{\kappa}_{2}=\frac{\lambda_{max}(P_{n}^{-1}A_{n})}{\lambda_{min}(P_{n}^{-1}A_{n})}<3.
\end{equation*}
\end{proof}

Corollary \ref{conditionnumber} manifests  that the spectrum of the preconditioned matrix is uniformly bounded independent of the matrix size, where the smallest eigenvalue is  away from $0$.  
Therefore, we conclude that the ${\tau}$-preconditioner is efficient for 1D Riesz fractional diffusion equations. Next, we extend our discussion on the multi-dimensional cases.

\subsection{Multi-dimensional cases}

Recalling the coefficient matrix ${\bf A}$ defined in (\ref{high-A}), the matrix $A_{n_i}$ is the coefficient matrix corresponding to 1D case. In the light of the results shown in previous subsection, we obtain the spectrum of the matrix $\tau(A_i)^{-1}A_{i}$ firstly.


\begin{lemma}\label{lmbda-PA}
 $A_{i}$ is a multi-level Toeplitz matrix defined in \eqref{Ai}. We then have for each $i=1,\dots,m$, the eigenvalues of ${\tau}(A_{i})^{-1}A_{i}$ satisfying
\begin{equation*}
\frac{1}{2}<\lambda({\tau(A_{i})^{-1}A_{i}})<\frac{3}{2}.
\end{equation*}
\end{lemma}
\begin{proof}
It is clear that
\begin{align*}
{\tau}(A_{i})^{-1}A_{i}&=\left(I_{n_{i}^{-}}\otimes {\tau({A}_{n_i})}\otimes I_{n_{i}^{+}}\right)^{-1}\left(I_{n_{i}^{-}}\otimes {A}_{n_i}\otimes I_{n_{i}^{+}}\right) \\
&=\left(I_{n_{i}^{-}}^{-1}\otimes {\tau({A}_{n_i})}^{-1}\otimes I_{n_{i}^{+}}^{-1}\right)
\left(I_{n_{i}^{-}}\otimes{A}_{n_i}\otimes I_{n_{i}^{+}}\right) \\
&=I_{n_{i}^{-}}\otimes ({\tau({A}_{n_i})}^{-1}A_{n_i})\otimes I_{n_{i}^{+}}.
\end{align*}
Invoking \cref{tau-spectrum}, for each $i$, it holds that
\begin{equation*}
\frac{1}{2}<{\lambda}({\tau}(A_{n_i})^{-1}A_{n_i})={\lambda}({\tau}(G^{(\alpha_i)}_{n_i})^{-1}G^{(\alpha_i)}_{n_i})<\frac{3}{2}.
\end{equation*}
Utilizing the properties of Kronecker product,
it is easy to check each $i$ holding
\begin{equation*}
\frac{1}{2}<\lambda({\tau(A_{i})^{-1}A_{i}})<\frac{3}{2}.
\end{equation*}
\end{proof}

This lemma shows that the spectrum of ${\tau}(A_i)^{-1}A_i$ is bounded for each $i$, which will be exploited to prove our aim conclusion that the spectrum of the preconditioned matrix ${\bP^{-1}\bA}$ is uniformly bounded.

\begin{theorem}\label{spectrum}
The spectrum of the preconditioned matrix ${\bP}^{-1}\bA$ is uniformly bounded below by $1/2$ and bounded above by $3/2$.
\end{theorem}
\begin{proof}
Let ${\bz}\in {\IR}^{N}$ be any nonzero vector. By the Rayleigh quotients theorem (see Theorem 4.2.2 in \cite{HJ-MA-2012}) and \cref{lmbda-PA}, for each $i$, it holds
\begin{equation*}
\frac{1}{2}<{\lambda}_{\min}({\tau(A_{i})^{-1}A_{i}})\le\frac{{\bz}^{\tT}A_{i}{\bz}}{{\bz}^{\tT}\tau(A_{i}){\bz}}
\le{\lambda}_{\max}({\tau(A_{i})^{-1}A_{i}})<\frac{3}{2},
\end{equation*}
that is,
\begin{equation*}
\frac{1}{2}{\bz}^{\tT}\tau(A_i){\bz}<{{\bz}^{\tT}A_{i}{\bz}}<\frac{3}{2}{\bz}^{\tT}\tau(A_i){\bz}.
\end{equation*}
Thus we have
\begin{equation*}
\frac{1}{2}{\bz}^{\tT}\sum\limits_{i=1}^{m}\tau(A_i){\bz}<{{\bz}^{\tT}\sum\limits_{i=1}^{m}A_{i}{\bz}}
<\frac{3}{2}{\bz}^{\tT}\sum\limits_{i=1}^{m}\tau(A_i){\bz}.
\end{equation*}
It results that
\begin{equation*}
\frac{1}{2}<\frac{{\bz}^{\tT}\sum\limits_{i=1}^{m}A_{i}{\bz}}{{\bz}^{\tT}\sum\limits_{i=1}^{m}\tau(A_i){\bz}}<\frac{3}{2},
\end{equation*}
i.e.,
\begin{equation*}
\frac{1}{2}<\frac{{\bz}^{\tT}{\bA}{\bz}}{{\bz}^{\tT}{\bP}{\bz}}<\frac{3}{2}.
\end{equation*}
Therefore, we have
\begin{equation*}
{\lambda}_{\min}({\bP}^{-1}{\bA})=\min_{\bz}\frac{{\bz}^{\tT}{\bA}{\bz}}{{\bz}^{\tT}{\bP}{\bz}}>1/2, \quad {\lambda}_{\max}({\bP}^{-1}{\bA})=\max_{\bz}\frac{{\bz}^{\tT}{\bA}{\bz}}{{\bz}^{\tT}{\bP}{\bz}}<3/2.
\end{equation*}
\end{proof}

 \cref{spectrum} indicates that the spectrum of the preconditioned matrix is uniformly bounded. Furthermore, applying the results of \cref{spectrum}, we obtain the following corollary.
\begin{corollary}
Let ${\bA}$ be the coefficient matrix of the multi-dimensional Riesz fractional diffusion equation defined in (\ref{high-A}), ${\bP}$ be the preconditioner defined in \eqref{high-preconditioner}. Then the condition number of the preconditioned matrix ${\bf P}^{-1}{\bf A}$ is less than $3$.
\end{corollary}


Up to now, for multi-dimensional Riesz fractional diffusion equations, we have proved that the spectrum of the preconditioned matrix is uniformly bounded. Since the smallest eigenvalue of the preconditioned matrix is bounded away from $0$, which manifests that the PCG method converges linearly. Moreover, we deduce that the condition number is less than 3, which reveals the number of iterations is independent of the size of the coefficient matrix. From the theoretical point of view, we have proved the efficiency of the $\tau$-preconditioner for ill-conditioned linear systems \eqref{linearsystem} and \eqref{high-linearsystem} arising from Riesz fractional derivative, whose generating function is with fractional order zeros.

\section{Extension to ill-conditioned multi-level Toeplitz matrices}\label{extension}
In this section, we extend our discussion to general multi-level Toeplitz matrices whose generating functions are with fractional order zeros at the origin.

Let ${\bafa}=(\alpha_1,\dots,\alpha_m)$ with each ${\alpha_i}\in (1,2)$, ${\btheta}=(\theta_1,\dots,\theta_m)\in [-\pi,\pi]^{m}$. Suppose $p_{\bafa}(\btheta)\in L^{1}([-\pi,\pi]^m)$ is a nonnegative integrable even function with $p_{\bafa}({\btheta})=0$ where ${\btheta}=({0,0,\dots,0})$.
Making use of $\eqref{Fouriercoefficient}$, $p_{\bafa}({\btheta})$ generates an $m$-level Toeplitz matrix ${\bB}$, which has the same structure as \eqref{m-level-matrix}, where the Toeplitz block at $i$-th level is with the matrix size $n_i\times n_i$ for $i=1,\dots,m$. Let $N=\prod\limits_{i=1}^{m}n_i$. We derive an $N\times N$ multi-level Toeplitz linear system
\begin{equation}\label{fsystem}
{\bB}{\bu}={\bb},
\end{equation}
where $\bu\in {\IR}^N$ is unknown, ${\bb}\in {\IR}^N$ is the right hand side.

As known that the inverse of the $\tau$ matrix  can be obtained in $\mO(N\log N)$ computational cost  by the fast discrete sine transform, we then construct a new $\tau$-preconditioner based on the discretized Riesz derivatives that can match the fractional order zero of the generating function.

Denote
\begin{equation}\label{generatingfunction_q}
q_{\bafa}({\btheta})=\sum_{i=1}^{m}l_{i}|{\theta_i}|^{\alpha_i},
\end{equation}
where $l_i>0$ for $i=1,\dots,m$ are   constants.
We refer to the multi-level Toeplitz matrix arising from $q_{\bafa}({\btheta})$ as ${\bQ}$.
For convenience of our investigation, the generating function $p_{\bafa}({\btheta})$ is required to satisfy the following assumption.
\begin{assumption}\label{assumption}
There exists two positive constants $c_{0}<c_{1}$ such that
\begin{equation*}
0<c_{0}\le\frac{p_{\bafa}({\btheta})}{q_{\bafa}({\btheta})}\le c_{1}.
\end{equation*}
\end{assumption}

 Note that this assumption indicates that the generating function $p_{\bafa}({\btheta})$ on each direction $i$ has an $\alpha_{i}$-th order zero at ${\theta_i}=0$ for $i=1,\dots,m$. According to precious analysis, we know that system \eqref{fsystem} is  symmetric positive definite and ill-conditioned. Therefore, a positive definite preconditioner is indispensable. However, the  positive definiteness of $\bB$ cannot guarantee that so is the corresponding ${\tau}$-preconditioner ${\tau(\bB)}$ \cite{S-MC-1999}. Therefore, the ${\tau(\bB)}$ cannot be directly applied to precondition ${ \bB }$. Actually, the numerical results shown in  \cref{t4} of \Cref{numericalresults}  exhibit the bad performance of the preconditioner ${\tau(\bB)}$.
Therefore, it is essential to design a preconditioner aiming at such kind of ill-conditioned systems arising from generating functions with fractional order zeros. In the following, we construct a new preconditioner based on the $\tau$-preconditioner of the Riesz fractional derivative that is symmetric positive definite.

Denote
\begin{equation}\label{G}
g_{\bafa}({\btheta})=\sum_{i=1}^{m}l_{i}g_{\alpha_i}({\theta_i}),
\end{equation}
where $g_{\alpha_i}({\theta_i})$ is defined as \eqref{g-form}.
Then, we obtain the corresponding matrix as
\begin{equation*}
{\bf G}=\sum\limits_{i=1}^{m}I_{n_{i}^{-}}\otimes l_{i}G_{n_i}^{(\alpha_i)}\otimes I_{n_{i}^{+}},
\end{equation*}
where $G_{n_i}^{(\alpha_i)}$ is defined as \eqref{A}.

Taking advantage of the above auxiliary tools, the specific procedures of constructing preconditioner $\bP$ of ${\bB}$ are depicted as follows.

\begin{itemize}
 \item [1.] The matrix ${\bQ}$ generated by $q_{\bafa}({\btheta})$ defined in \eqref{generatingfunction_q} is employed to approximate ${\bB}$, denoted as    ${\bP}^{(1)}={\bQ}$.
 \item [2.] The matrix ${\bG}$ generated by \eqref{G} is exploited to approximate $\bQ$, denoted as ${\bP}^{(2)}={\bG}{\bQ}^{-1}$.
 \item [3.] The matrix ${\tau({\bG})}$ is used to approximate ${\bG}$, denoted as ${\bP^{(3)}}= \tau({\bG}){\bG}^{-1}$.
 \item [4.] Finally, the preconditioner of $\bB$ is constructed by $\bP={\bP}^{(3)}{\bP}^{(2)}{\bP}^{(1)}$; i.e.,
\begin{equation}\label{new_pre}
{\bP}={\tau}(\bf{G}).
\end{equation}
\end{itemize}

To obtain the spectrum of the preconditioned matrix $\tau({\bG})^{-1}{\bB}$, the properties of matrices ${\bP}^{(1)}$, ${\bP}^{(2)}$, ${\bP}^{(3)}$ will be considered. Under \cref{assumption}, the following lemma is achieved.

\begin{lemma}\label{BQ}
Let ${\bf B}$, ${\bf Q}$ be the multi-level Toeplitz matrices generated by $p_{\bafa}({\btheta})$ and $q_{\bafa}({\btheta})$, respectively. Then, for any nonzero vector ${\bz}\in {\IR}^{N}$, we have
\begin{equation*}
\frac{{\bz}^{\tT}{\bB}{\bz}}{{\bz}^{\tT}{\bf Q}{\bz}}\in [c_{0},c_{1}].
\end{equation*}
\end{lemma}
\begin{proof}
Note that the eigenvalues of matrix ${\bQ^{-1}\bB}$ are subjected to the Grenander--Szeg\"{o}'s theorem \cite{GS-Toepliz-form}; i.e.,
\begin{equation*}
\min\frac{p_{\bafa}(\btheta)}{q_{\bafa}(\btheta)}\le\lambda_{\min}(\bQ^{-1}\bB)\le\lambda_{\max}(\bQ^{-1}\bB)\le
\max\frac{p_{\bafa}(\btheta)}{q_{\bafa}(\btheta)}.
\end{equation*}
Combining the Rayleigh quotients theorem with \cref{assumption}, we obtain
\begin{equation*}
0<c_{0}\le\lambda_{\min}(\bQ^{-1}\bB)\le\frac{\bz^{\tT}{\bf B}\bz}{\bz^{\tT}{\bf Q}\bz}\le\lambda_{\max}(\bQ^{-1}\bB)\le c_{1}.
\end{equation*}
The proof is complete.
\end{proof}

Similarly, we have the following lemma.
\begin{lemma}\label{GQ}
The matrix ${\bf G}$ is generated by $g_{\bafa}({\btheta})$ defined in \eqref{G}. For ${\alpha_i}\in(1,2)$, $i=1,\dots,m$, it holds that
\begin{equation*}
\frac{1}{2}\le\lambda({\bf {G}^{-1}Q})\le c_2,
\end{equation*}
where $c_2=\max\limits_{i}\frac{\pi^{2}}{-8\cos({\pi\alpha_i/2})}$.
\end{lemma}
\begin{proof}
It is easy to check that
\begin{equation*}
\min\limits_{i}\frac{|\theta_i|^{\alpha_i}}{g_{\alpha_i}({\theta_i})}\le\frac{q_{\bafa}(\btheta)}{g_{\bafa}(\btheta)}
=\frac{\sum\limits_{i=1}^{m}l_{i}|\theta_{i}|^{\alpha_i}}
{\sum\limits_{i=1}^{m}l_{i}g_{\alpha_i}(\theta_i)}\le\max\limits_{i}\frac{|\theta_i|^{\alpha_i}}{g_{\alpha_i}({\theta_i})}.
\end{equation*}
From \cref{fractional-zeros-bounded}, we derive that for each $i$ holds
\begin{equation*}
\frac{1}{2}\le\frac{|\theta_i|^{\alpha_i}}{g_{\alpha_i}({\theta_i})}\le\frac{\pi^{2}}{-8\cos({\pi\alpha_i/2})}.
\end{equation*}
Then, it immediately arrives
\begin{equation*}
\frac{1}{2}\le\frac{q_{\bafa}(\btheta)}{g_{\bafa}(\btheta)}\le\max\limits_{i}\frac{\pi^{2}}{-8\cos({\pi\alpha_i/2})}=c_2.
\end{equation*}
In view of the proof of \cref{BQ}, it suffices to show that
\begin{equation*}
\frac{1}{2}\le\min\frac{g_{\bafa}(\btheta)}{q_{\bafa}(\btheta)}\le\lambda({\bf {G}^{-1}Q})\le\max\frac{g_{\bafa}(\btheta)}{q_{\bafa}(\btheta)}\le c_2.
\end{equation*}
\end{proof}

With the help of the above lemmas, the following theorem with respect to the spectrum of the preconditioned matrix is attained.

\begin{theorem}\label{spectrum-bounded}
The spectrum of preconditioned matrix ${\tau}(\bf G)^{-1}{\bf B}$ is bounded below and above by constants independent of the matrix size; i.e.,
\begin{equation}\label{preconditioned-bounded}
\frac{c_{0}}{4}\le{\lambda}({\tau}({\bf G})^{-1}{\bf B})\le\frac{3c_{1}c_{2}}{2}.
\end{equation}
\end{theorem}

\begin{proof}
Note that by \cref{spectrum}, for any nonzero vector ${\bz}\in{\IR}^{N}$, it follows that
\begin{equation*}
\frac{1}{2}<\frac{\bz^{\tT}{\bf G}\bz}{\bz^{\tT}\tau({\bf G})\bz}<\frac{3}{2}.
\end{equation*}
Then, combining \cref{BQ} and \cref{GQ}, it holds that
\begin{equation*}
\frac{c_{0}}{4}\le\frac{\bz^{\tT}{\bf B}\bz}{\bz^{\tT}\tau({\bf G})\bz}=\frac{\bz^{\tT}{\bf G}\bz}{\bz^{\tT}\tau({\bf G})\bz}{\cdot}
\frac{\bz^{\tT}{\bf Q}\bz}{\bz^{\tT}{\bf G}\bz}{\cdot}\frac{\bz^{\tT}{\bf B}\bz}{\bz^{\tT}{\bf Q}\bz}\le\frac{3c_{1}c_{2}}{2}.
\end{equation*}
Invoking the Rayleigh quotients theorem again, we derive
\begin{equation*}
\frac{c_{0}}{4}\le{\lambda}({\tau}({\bf G})^{-1}{\bf B})\le\frac{3c_{1}c_{2}}{2}.
\end{equation*}
\end{proof}

This theorem implies the spectrum of the preconditioned matrix is bounded, and the smallest eigenvalue of the preconditioned matrix is bounded away from $0$, which means the linear convergent rate if the PCG method is exploited to solve ill-conditioned multi-level Toeplitz systems. Moreover, we derive the following corollary.

\begin{corollary}
The condition number of the preconditioned matrix ${\tau}({\bf G})^{-1}{\bf B}$ is bounded by constant.
\end{corollary}
\begin{proof}
From \cref{spectrum-bounded}, we obtain
\begin{equation*}
{\kappa}_{2}(\tau(\bG)^{-1}\bB)=\frac{\lambda_{\max}(\tau(\bG)^{-1}\bB)}{{\lambda}_{\min}(\tau(\bG)^{-1}\bB)}
\le\frac{6c_{1}c_{2}}{c_0}.
\end{equation*}
\end{proof}

This corollary indicates the condition number of the preconditioned matrix is bounded by a constant. Hence, the number of iterations of the PCG method with our preconditioner for solving the system \eqref{fsystem} is expected to be independent of the matrix size.

\section{Numerical results}
\label{numericalresults}
In this section, the numerical experiments are carried out to examine the efficiency of the proposed methods. All results are performed via MATLAB R2017a on a PC with the
configuration: Intel(R) Core(TM)i7-8700 CPU $@$3.20 3.20GHz and 8 GB RAM.

To exhibit the performance of the PCG method with the ${\tau}$-preconditioner for solving linear systems \eqref{linearsystem} and \eqref{high-linearsystem}, we also implement  other preconditioners as comparisons, including the Strang circulant preconditioner \cite{LS-JCP-2012} and the banded preconditioner \cite{LYJ-CCP-2015}, where
 the banded preconditioner   $B_{n}$ is defined as
\begin{equation*}{\footnotesize
B_{n}=
\left[
\begin{array}{cccccc}
b_{0} & b_{1}  & \dots& b_{k-1}  & \  &  \   \\
b_{1}&   b_{0} & a_{1} &  \ & \ddots  & \    \\
\vdots & b_{1} &\ddots & \ &\  & b_{k-1}   \\
b_{k-1} & \ & \ddots & \ddots &\ & \vdots \\
\  & \ddots  & \ & \ddots &\ddots  &b_{1} \\
\ & \ & b_{k-1} & \dots & b_{1}  &b_{0} \\
\end{array}
\right]
+\left[
\begin{array}{cccccc}
0 & \  & \ & \  & \  &  \   \\
\    &  \ddots &\ &  \ & \  & \    \\
\ & \ &  0 & \ &\  & \   \\
\ & \ & \ & 2*b_{k} &\ & \ \\
\  & \  & \ & \    &\ddots  &\ \\
\ & \ & \ & \ & \  &  2*\sum\limits_{j=k}^{n-1}b_{j} \\
\end{array}
\right] }
\end{equation*}
with `$k$' being the bandwidth of the preconditioner. In practical computation, we choose $k=8$. Besides, the multigrid methods \cite{SJC-BIT-2001,DKMT-AdvCM-2020} are also presented in our numerical experiments.

In the following tables, `${\tau}_{pre}$', `$C_{pre}$', `$B_{pre}$' and `$N_{pre}$' represent the PCG method with the $\tau$-preconditioner, the Strang circulant preconditioner, the banded preconditioner, and no preconditioner, respectively. `$M_{pre}$' denotes the multigrid preconditioned method with the banded preconditioner proposed in \cite{DKMT-AdvCM-2020} and `$MGM$' represents the algebraic multigrid method shown in \cite{SJC-BIT-2001}.
Besides, `$n$' denotes the spatial grid points, `Iter' displays the number of iterations required for convergence by those methods, and `CPU(s)' signifies the CPU time in second for solving the linear systems. In particular, `$-$' implies the CPU time over $10^4$ and `$*$' represents the number of iterations over $10^3$.
For all tested methods, let $u_{0}=0$ be the initial guess and the stopping criterion is chosen as
\begin{equation*}
\frac{\left\|r_{q}\right\|_2}{\left\|r_{0}\right\|_2}<10^{-8},
\end{equation*}
where $r_{q}$ is the residual vector at the $q$-th iteration.

\begin{example}\label{ex1}
Consider 1D Riesz fractional diffusion equations defined in $[0,1]$. Take the diffusion coefficient $d=1$, and the source term is determined by the exact solution, which is given by $u(x)=x^{2}(1-x)^{2}$.
\end{example}

The number of iterations by those iterative methods for solving  \cref{ex1} are displayed in \cref{t1}. It is obvious that both the $\tau$-preconditioner and the circulant preconditioner show good performance, while the banded preconditioner needs more iterations.
On the other hand, \cref{fig} exhibits the spectral distribution of the coefficient matrix and the preconditioned matrices with ${\alpha}=1.2$ and $n=2^{10}-1$, where the values on the $x$-axis direction represent the logarithmic values of eigenvalues. We observe that the spectrum of the coefficient matrix $A$ is scattered on the coordinate axis from the left figure of \cref{fig} and hence more iterations are required by the CG method to converge.
The comparisons of the spectral distribution of the preconditioned matrices with different preconditioners are shown in the right one,
where
`$\tau^{-1}A$', `$B^{-1}A$', and `$C^{-1}A$' denote the preconditioned matrices with the $\tau$-preconditioner, the banded preconditioner, and the circulant preconditioner, respectively. Note that the $\tau$-preconditioner has a highly clustered spectrum, which illustrates the better performance of the $\tau$-preconditioner.

Moreover, we list the extreme eigenvalues of $\tau^{-1}A$ for $\alpha=1.8$ in \cref{tc}. We derive that all the eigenvalues are located in the open interval $(1/2,3/2)$,
which is coincident with our theoretical analysis. Those numerical results exemplify the efficiency of the proposed preconditioner.

\begin{table}[!htbp]
	\setlength{\belowcaptionskip}{10pt}
	\renewcommand{\arraystretch}{1.25}
	\centering
	\caption{Comparisons of the iterations for solving \cref{ex1} for different ${\alpha}$ by the CG method, and the PCG methods with the ${\tau}$-preconditioner, the circulant preconditioner, and the banded preconditioner.}
	\label{t1}

	\resizebox{\textwidth}{!}{%
{\footnotesize
		\begin{tabular}{c|ccrr|ccrr|ccrr}
			\toprule
\multirow{2}{*}{$n_{1}+1$}
 &   \multicolumn{4}{c}{$\alpha=1.2$} & \multicolumn{4}{c}{$\alpha=1.5$} & \multicolumn{4}{c}{$\alpha=1.8$} \\
\cline{2-5}       \cline{6-9}  \cline{10-13}
  &$\tau_{pre}$ &$C_{pre}$ &$B_{pre}$ &$N_{pre}$  &$\tau_{pre}$ &$C_{pre}$ &$B_{pre}$ &$N_{pre}$                  &$\tau_{pre}$ &$C_{pre}$ &$B_{pre}$ &$N_{pre}$ \\
  \midrule
 $2^6$         & 5      & 5   &  9        &  32   & 5   & 5   & 9    & 32   & 4    & 5   &  7   & 32 \\
 $2^7$         & 5      & 5   &  12       &  63   & 5   & 5   & 11   & 62   & 5    & 6   &  8   & 64 \\
 $2^8$         & 5      & 6   &  16       & 110   & 5   & 7   & 14   & 111  & 5    & 7   &  10  & 126 \\
 $2^9$         & 6      & 6   &  20       & 178   & 6   & 7   & 17   & 192  & 5    & 7   &  11  & 238  \\
 $2^{10}$      & 6      & 6   &  26       & 279   & 6   & 8   & 21   & 328  & 6    & 7   &  13  &448  \\ \bottomrule
	\end{tabular}}}
\end{table}

\begin{figure}[!htbp]
\centering
\begin{minipage}[t]{0.48\textwidth}
{\includegraphics[width=6.2cm]{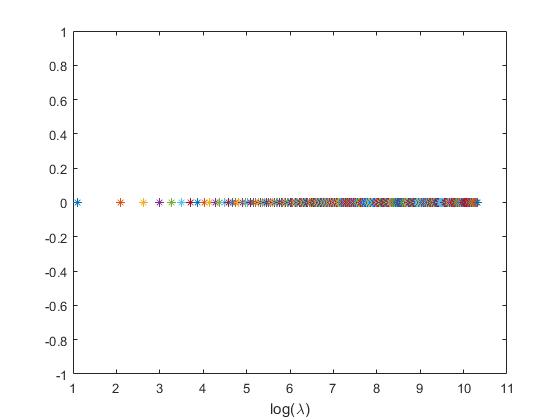}}
\end{minipage}
\begin{minipage}[t]{0.48\textwidth}
{\includegraphics[width=6.2cm]{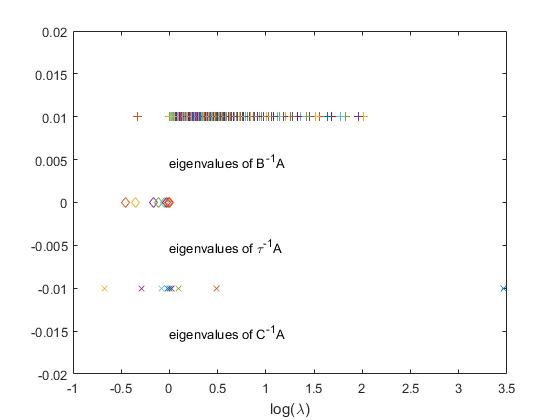}}
\end{minipage}
\caption{The spectral distribution of coefficient matrix $A$ and preconditioned matrices $P^{-1}A$.}
\label{fig}
\end{figure}

\begin{table}[!htbp]
	\setlength{\belowcaptionskip}{10pt}
	\renewcommand{\arraystretch}{1.25}
	\centering
\caption{Extreme eigenvalues of preconditioned matrix ${\tau}^{-1}A$ of \cref{ex1} with $\alpha=1.8$.}\label{tc}
\resizebox{0.8\textwidth}{!}{%
{\footnotesize
		\begin{tabular}{cccccccc}
			\toprule
n  & $2^6$ & $2^7$ & $2^8$ & $2^9$ & $2^{10}$ & $2^{11}$ & $2^{12}$ \\
\hline
$\lambda_{\max}$  & 1.0001 & 1.0001 & 1.0001 &1.0001 & 1.0001 & 1.0001 & 1.0001  \\
$\lambda_{\min}$  & 0.8721 & 0.8586 & 0.8473 & 0.8379 & 0.8300 & 0.8232 & 0.8173 \\
 \bottomrule
	\end{tabular}}}
\end{table}

\begin{example}\label{ex2}
In this example, we consider $2D$ Riesz fractional diffusion equations in ${\Omega}\in [0,1]\times[0,1]$, where the exact solution is
\begin{equation*}
u(x_{1},x_{2})=x_{1}^{2}(1-x_{1})^{2}{x_{2}}^{2}(1-x_{2})^{2}.
\end{equation*}
Take the diffusion coefficients $d_{1}=d_{2}=1$. The source term is given by
\begin{align*}
y=&\frac{d_1}{\cos(\pi\alpha_1/2)}x_{2}^{2}(1-x_{2})^{2}(y_{1}(x_{1},\alpha_1)+y_{1}(1-x_{1},\alpha_1)) \\
&+\frac{d_2}{\cos(\pi\alpha_2/2)}x_{1}^{2}(1-x_{1})^{2}(y_{1}(x_{2},\alpha_2)+y_{1}(1-x_{2},\alpha_2)),
\end{align*}
where
\begin{equation*}
y_{1}(x,\alpha)=\frac{2}{\Gamma(3-\alpha)}x^{2-\alpha}-\frac{12}{\Gamma(4-\alpha)}x^{3-\alpha}+
\frac{24}{\Gamma(5-\alpha)}x^{4-\alpha}.
\end{equation*}
\end{example}

\begin{table}[t]
	\setlength{\belowcaptionskip}{10pt}
	\renewcommand{\arraystretch}{1.25}
	\centering
	\caption{Comparisons for solving \cref{ex2} for different ${\alpha}$ by the CG method, and the PCG methods with the ${\tau}$-preconditioner, and the circulant preconditioner, and multigrid methods.}
	\label{t2}

	\resizebox{\textwidth}{!}{%
{\footnotesize
		\begin{tabular}{c|c|cr|cr|rr|rr|rr}
			\toprule
			\multirow{2}{*}{$(\alpha_{1},\alpha_{2})$} &\multirow{2}{*}{$n_{1}+1$}  & \multicolumn{2}{c}{${\tau}_{pre}$} &\multicolumn{2}{c}{$C_{pre}$} &\multicolumn{2}{c}{$M_{pre}$}\ &\multicolumn{2}{c}{$MGM$} &\multicolumn{2}{c}{$N_{pre}$} \\
		\cline{3-4}\cline{5-6}\cline{7-8}\cline{9-10}\cline{11-12}
	&	&Iter&CPU(s)&Iter&CPU(s)&Iter&CPU(s)&Iter&CPU(s)&Iter&CPU(s)\\
			\midrule
(1.1,1.2) & $2^6$  & 7 & 0.11  & 17 & 0.39  &13 &0.64  &13 &0.09   &93  & 0.75   \\
          & $2^7$  & 7 & 0.21  & 19 & 0.45  &14 &0.89  &11 &0.45   &157 & 3.63   \\
          & $2^8$  & 8 & 1.37  & 21 & 2.17  &17 &4.42  &10 &2.89   &237 & 12.34  \\
          & $2^9$  & 8 & 5.06  & 24 & 9.84  &22 &12.20 &9  &5.59   &383 & 82.59  \\
       & $2^{10}$  & 9 &13.03  & 27 & 29.56 &31 &80.05 &9  &24.33  &585 & 490.56 \\

\hline
(1.4,1.5) & $2^6$  & 7 & 0.22  & 16 & 0.33  &12 &0.72 &12 &0.50   &  91  & 0.78         \\
          & $2^7$  & 7 & 0.31  & 19 & 0.45  &13 &0.86 &11 &0.80   &  157 & 3.25          \\
          & $2^8$  & 8 & 1.08  & 23 & 2.08  &16 &3.78 &12 &3.14   &  269 & 14.23         \\
          & $2^9$  & 8 & 4.34  & 28 & 11.78 &23 &13.27&12 &7.56   &  457 & 90.08         \\
       & $2^{10}$  & 9 & 12.56 & 32 & 34.11 &36 &96.06&13 &31.25  &  771 & 615.25        \\

\hline
(1.8,1.9) & $2^6$     & 6   & 0.10   & 19 & 0.53 &11 &0.59 &16 &0.11  &126  &0.86           \\
          & $2^7$     & 6   & 0.14   & 24 & 1.05 &12 &0.75 &15 &0.72  &243  &5.28           \\
          & $2^8$     & 7   & 0.67   & 31 & 2.84 &12 &4.28 &15 &3.77  &467  &27.97          \\
          & $2^9$     & 7   & 2.92   & 40 & 15.91&22 &12.45&16 &9.64  &901  &178.20         \\
          & $2^{10}$  & 7   & 10.42  & 52 & 54.52&34 &85.94&17 &40.53 &1740 &1383.60        \\
\hline
(1.2,1.8) & $2^6$     & 6   & 0.17   & 19  & 0.27 &9  &0.58  &42  &0.39  &  127  & 1.05       \\
          & $2^7$     & 7   & 0.31   & 27  & 0.83 &10 &0.73  &60  &2.50  &  247  & 5.06       \\
          & $2^8$     & 7   & 1.22   & 33  & 3.20 &14 &3.94  &87  &19.28 &  463  & 24.30       \\
          & $2^9$     & 8   & 4.05   & 44  & 18.80&21 &12.44 &126 &73.89 &  881  & 173.53     \\
          & $2^{10}$  & 8   & 11.56  & 58  & 61.42&30 &73.19 &184 &402.94&  1671 & 1352.76     \\ \bottomrule
	\end{tabular}}}
\end{table}


In the 2D case, we exploit the multigrid preconditioned method and the algebraic multigrid method for comparisons, where the weight of the Jacobi iterative method is chosen as $w=2/3$.


In this example, we take $n_{1}=n_{2}$.  From \cref{t2},  note that both  of the number of iterations and the CPU time provided by the $\tau$-preconditioner are much less than those by other methods. It demonstrates    the superiority of the $\tau$-preconditioner.

\begin{example}\label{ex3}
In this example, we test 3D Riesz fractional diffusion equations. Consider
\begin{equation*}
{\Omega}=[0,1]\times[0,1]\times[0,1],\ d_{1}=d_{2}=d_{3}=1,
\end{equation*}
\begin{align*}
y=&\frac{d_{1}}{\cos(\pi\alpha_{1}/2)}x_{2}^{2}(1-x_{2})^{2}x_{3}^{2}(1-x_{3})^{2}(y_{1}(x_{1},\alpha_1)
+y_{1}(1-x_{1},\alpha_1)) \\
&+\frac{d_{2}}{\cos(\pi\alpha_{2}/2)}x_{1}^{2}(1-x_1)^{2}x_{3}^{2}(1-x_{3})^{2}(y_{1}(x_{2},\alpha_2)+y_{1}(1-x_{2},\alpha_2)) \\
&+\frac{d_{3}}{\cos(\pi\alpha_{3}/2)}x_{1}^{2}(1-x_{1})^{2}x_{2}^{2}(1-x_{2})^{2}(y_{1}(x_{3},\alpha_3)+y_{1}(1-x_{3},\alpha_3)).
\end{align*}
The exact solution is $u=x_{1}^{2}(1-x_{1})^{2}{x_{2}}^{2}(1-x_{2})^{2}{x_{3}}^{2}(1-x_{3})^{2}$.
\end{example}

\begin{table}[!htbp]
	\setlength{\belowcaptionskip}{10pt}
	\renewcommand{\arraystretch}{1.25}
	\centering
	\caption{Comparisons for solving \cref{ex3} for different ${\alpha}$ by the CG method, and the PCG methods with the ${\tau}$-preconditioner and the circulant preconditioner.}
	\label{t3}

	\resizebox{0.85\textwidth}{!}{%
{\footnotesize
		\begin{tabular}{c|c|cr|cr|rr}
			\toprule
			\multirow{2}{*}{$(\alpha_{1},\alpha_{2},\alpha_{3})$} &\multirow{2}{*}{$n_{1}+1$}  & \multicolumn{2}{c}{${\tau}_{pre}$} &\multicolumn{2}{c}{$C_{pre}$} &\multicolumn{2}{c}{$N_{pre}$}\\
		\cline{3-4}\cline{5-6}\cline{7-8}
	&	&Iter&CPU(s)&Iter&CPU(s)&Iter&CPU(s)\\
			\midrule
(1.1,1.2,1.3) & $2^4$   & 6     & 0.19   & 14 & 0.31     &  40   & 0.28     \\
              & $2^5$   & 6     & 0.73   & 17 & 1.41     &  70   & 3.67     \\
              & $2^6$   & 7     & 5.34   & 21 & 11.88    &  118  & 36.27    \\
              & $2^7$   & 8     & 29.39  & 24 & 69.69    &  191  & 288.64   \\
              & $2^8$   & 8     & 226.28 & 27 & 521.38   &  304  & 2.98e+3  \\
\hline
(1.4,1.5,1.6) & $2^4$      & 6  & 0.22    & 15 & 0.40     &  39  & 0.19     \\
              & $2^5$      & 7  & 0.48    & 18 & 1.42     &  71  & 3.69     \\
              & $2^6$      & 7  & 4.77    & 22 & 11.03    &  128 & 36.64    \\
              & $2^7$      & 7  & 25.84   & 25 & 72.28    &  223 & 335.20   \\
              & $2^8$      & 8  & 200.64  & 32 & 611.23  &  387 & 3.77e+3   \\
\hline
(1.7,1.8,1.9) & $2^4$     & 5   & 0.23   & 16 & 0.45   &  45  & 0.25        \\
              & $2^5$     & 6   & 0.64   & 20 & 1.59   &  88  & 4.06        \\
              & $2^6$     & 6   & 3.93   & 26 & 13.88  &  169 & 50.38       \\
              & $2^7$     & 6   & 24.34  & 35 & 95.42  &  328 & 495.33      \\
              & $2^8$     & 7   & 180.61 & 44 & 832.36  & 628 & 6.08e+3     \\
\hline
(1.2,1.5,1.8) & $2^4$     & 6   & 0.25   & 16 & 0.36   &  43  & 0.23        \\
              & $2^5$     & 6   & 0.64   & 20 & 1.63   &  83  & 3.73        \\
              & $2^6$     & 7   & 4.95   & 25 & 13.70  &  157 & 48.13       \\
              & $2^7$     & 8   & 18.11  & 33 & 94.09  &  295 & 441.98      \\
              & $2^8$     & 8   & 198.59 & 44 & 828.98 & 551  &5.34e+3      \\  \bottomrule
	\end{tabular}}}
\end{table}

In this example, let $n_{1}=n_{2}=n_{3}$. It is evident that the ${\tau}$-preconditioner is still efficient for 3D Riesz fractional diffusion equations from the \cref{t3}.
In contrast with the circulant preconditioner, the number of iterations by the ${\tau}$-preconditioner is almost unchanged when the matrix size increases. Hence, ${\tau}$-preconditioner is an excellent tool for solving ill-conditioned multi-level Toeplitz systems arising from multi-dimensional Riesz fractional diffusion equations.

Finally, we verify the effectiveness of the preconditioning of discretized Riesz fractional derivatives for handling the ill-conditioned multi-level Toeplitz matrices whose generating functions are with fractional order zeros at the origin.

\begin{example}\label{ex4}
Consider a two-level Toeplitz matrix whose generating function is defined by
\begin{equation*}
p_{\bafa}(\btheta)=p_{\alpha_1}(\theta_1)+p_{\alpha_2}(\theta_2)-p_{1}(\theta_1)p_{1}(\theta_2),
\end{equation*}
where {\rm (see \cite{RS-JSC-1998})}
\begin{equation*}
p_{\alpha_i}({\theta_i})=\left\{
\begin{array}{cc}
|{\theta_i}|^{\alpha_i}, & |{\theta_i}|<\frac{\pi}{2}, \\
1, & |{\theta_i}|\ge\frac{\pi}{2}.
\end{array}
\right.
\end{equation*}
\end{example}

It is obvious that
\begin{equation*}
0<\frac{4-\pi}{4}\le\frac{p_{\bafa}(\btheta)}{q_{\bafa}(\btheta)}\le1,
\end{equation*}
where $q_{\bafa}(\btheta)$ is defined in \eqref{GQ} with $l_1=l_2=1$. The above inequality exemplifies that the generating function $p_{\bafa}(\btheta)$ satisfies the \cref{assumption}.

\begin{table}[!htbp]
	\setlength{\belowcaptionskip}{10pt}
	\renewcommand{\arraystretch}{1.25}
	\centering
	\caption{Comparisons for solving \cref{ex4} for different ${\alpha}$ by the CG method, and the PCG methods with the ${\tau}$-preconditioner, the circulant preconditioner, and our preconditioner, and the algebraic multigrid method.}
	\label{t4}
{\footnotesize
	\resizebox{\textwidth}{!}{%
		\begin{tabular}{c|c|cr|rr|rr|rr|rr}
			\toprule
			\multirow{2}{*}{$(\alpha_{1},\alpha_{2})$} &\multirow{2}{*}{$n_{1}+1$}  &\multicolumn{2}{c}{$R_{pre}$} & \multicolumn{2}{c}{${\tau}_{pre}$} &\multicolumn{2}{c}{$C_{pre}$} &\multicolumn{2}{c}{$N_{pre}$} &\multicolumn{2}{c}{$MGM$}   \\
		\cline{3-4}\cline{5-6}\cline{7-8}\cline{9-10} \cline{11-12}
	&	&Iter&CPU(s)&Iter&CPU(s)&Iter&CPU(s)&Iter&CPU(s)&Iter&CPU(s)\\
			\midrule
(1.9,1.5)&$2^6$      & 26  & 0.42   & 13  & 0.23   &26    & 0.31    &79   & 0.36  &9  &0.22 \\
          & $2^7$    & 26  & 0.92   & 13  & 0.44   &36    & 1.17    &134  & 2.95  &10 &0.67 \\
          & $2^8$    & 27  & 4.97   & 16  & 2.58   &45    & 4.81    &225  & 15.69 &12 &2.36  \\
          & $2^9$    & 27  & 15.41  & 19  & 12.77  &69    & 35.14   &376  & 136.80&14 &11.75  \\
       & $2^{10}$    & 27  & 46.48  & 25  & 50.16  &192   & 278.50  &566  & 572.09&18 &58.38 \\
       & $2^{11}$    & 27  & 188.77 & 42  & 302.30 &239   & 1.44e+3 &956  &4.19e+3&21 &329.05 \\
       & $2^{12}$    & 27  & 659.05 & 57  & 1.36e+3&782   & --      &*    &  --   &26 &1.71e+3  \\
\hline
(1.9,1.7)& $2^6$    & 26   & 0.53    & 14 & 0.14   & 29 & 0.63     &90    &0.42   &9  &0.22  \\
          & $2^7$   & 26   & 0.84    & 17 & 0.63   & 44 & 1.34     &159   & 4.44  &9  &0.77  \\
          & $2^8$   & 26   & 3.95    & 20 & 2.98   & 79 & 8.02     &278   & 20.34 &10 &2.24  \\
          & $2^9$   & 26   & 15.47   & 30 & 19.05  & 145 & 73.41   &467   & 132.67&11 &9.77  \\
       & $2^{10}$   & 27   & 47.89   & 49 & 83.78  & 270 & 385.14  &846   & 907.67&12 &48.60  \\
       & $2^{11}$   & 27   & 188.72  & 85 &565.31  & 736 &5.20e+3  &*     & --    &13 &216.62   \\
       & $2^{12}$   & 27   & 690.84  &219 &5.09e+3 &*   & --       &*     &  --   &15 &1.01e+3  \\

\hline
(1.9,1.9) & $2^6$     & 27   & 0.39   & 15 & 0.19   & 30  & 0.27   &97   & 0.64     &9 &0.33   \\
          & $2^7$     & 27   & 0.91   & 21 & 0.93   & 57  & 1.77   &179  & 3.56     &9 &0.64   \\
          & $2^8$     & 27   & 4.17   & 26 & 4.12   & 85  & 8.58   &327  & 22.75    &9 &2.33   \\
          & $2^9$     & 27   & 17.27  & 44 & 26.34  & 212 &111.33  &573  & 210.27   &9 &8.41   \\
          & $2^{10}$  & 27   & 47.45  & 80 & 135.02 & 622 &884.95   &943  & 948.84  &9 &37.30  \\
          & $2^{11}$  & 27   & 190.06 & 195&1.29e+3 & *   &--       &*   & --       &10&166.47 \\
          & $2^{12}$  & 27   & 660.47 &530 &--      & *   &--       &*   &  --      &10&683.06 \\
			\bottomrule
	\end{tabular}}}
\end{table}

To test the efficiency of our proposed preconditioner defined in \eqref{new_pre}, the PCG methods with the ${\tau}$-preconditioner, the circulant preconditioner, and the case without preconditioner are proposed as comparisons. Denote our preconditioner as `$R_{pre}$' in \cref{t4}.

From \cref{t4}, we see that the number of iterations deriving from the circulant preconditioner, the natural ${\tau}$-preconditioner, and no preconditioner increase rapidly as the matrices size increase, while that by our proposed preconditioner almost keeps constant. In terms of the multigrid method, we observe that the method implements well when the orders of the zeros are closed. Since the cost per iteration of using the algebraic multigrid method is about $8/3$ times than the required by the $\tau$-preconditioner \cite{SJC-BIT-2001},
in light of \cref{t4}, we derive that the performance of the multigrid method is as efficient as the proposed method. Nevertheless, the multigrid method will be inefficient provided that the orders of the zeros are quite difference; see the item for $(\alpha_1,\alpha_2)=(1.9,1.5)$ in \cref{t4}. Moreover, the linearly convergent rate is still a question for these cases by the algebraic multigrid method.


\section{Conclusion remarks}\label{conclusion}
In this paper, we have studied the spectra of the $\tau$-preconditioned matrices for  the multi-level Toeplitz systems
arising from the multi-dimensional Riesz spatial fractional diffusion equations. Theoretically, we have proved that the spectra of the preconditioned matrices are bounded
below by 1/2 and bounded above by 3/2, and hence  the condition numbers of the preconditioned matrices are all less than $3$. Besides, we  proposed
a new preconditioner for ill-conditioned multi-level Toeplitz systems, which are generated by the generating functions with fractional order zeros at the origin.  We have  proved that the spectra of the proposed preconditioned matrices are bounded by constants which are independent of the matrices size. The numerical results have revealed that
the performance of the proposed preconditioner is much better than  that of other existing methods.
In our future work, we will
consider to combine the ${\tau}$-preconditioner with other methods to handle more general (non-symmetric) multi-level Toeplitz-like
systems arising from multi-dimensional fractional partial differential equations.


\end{document}